\documentclass[reqno]{amsart}
\usepackage{amsfonts}
\usepackage{graphicx}
\usepackage{amscd}
\usepackage{amsmath}
\usepackage{amssymb}
\usepackage[latin1]{inputenc}
\usepackage{amssymb,amsmath}
\usepackage{amsfonts}
\usepackage{graphicx,color}
\usepackage{amsthm}
\usepackage{graphicx}
\usepackage{fancyhdr}
\usepackage{hyperref}
\usepackage{color}
\theoremstyle{plain}

\newtheorem{theorem}{\bf Theorem}
\newtheorem{corollary}{\bf Corollary}

\newtheorem{lemma}{\bf Lemma}

\newtheorem{proposition}{\bf Proposition}
\newtheorem{remark}{Remark}

\numberwithin{equation}{section}

\newcommand\dv{\mathrm{div}}
\newcommand\dMt{\mathrm{dv_{g(t)}}}
\newcommand\dM{\mathrm{dv_g}}
\newcommand\dnt{\mathrm{d\sigma_{g(t)}}}
\newcommand\dn{\mathrm{d\sigma_g}}
\newcommand\dg{\mathrm{dv_{g_0}}}
\newcommand\db{\mathrm{d\sigma_{g_0}}}
\begin{document}

\title[On Eigenvalue Generic Properties of the Laplace-Neumann Operator]{On Eigenvalue Generic Properties of the Laplace-Neumann Operator}
\author{Jos\'e N.V. Gomes$^1$}
\author{Marcus A.M. Marrocos$^2$}

\address{$^1$Departamento de Matem\'atica, Universidade Federal do Amazonas, Av. Rodrigo Oct\'avio, 6200, 69080-900 Manaus, Amazonas, Brazil}
\address{$^2$CMCC-Universidade Federal do ABC, Av. dos Estados, 5001, 09210-580 Santo Andr\'e, S\~ao Paulo, Brazil.}

\email{$^1$jnvgomes@gmail.com; jnvgomes@pq.cnpq.br}
\email{$^2$marcus.marrocos@ufabc.edu.br}

\urladdr{$^1$https://ufam.edu.br; http://ppgm.ufam.edu.br}
\urladdr{$^2$http://cmcc.ufabc.edu.br}

\thanks{$^1$Partially supported by grant 202234/2017-7, Conselho Nacional de Desenvolvimento Cient\'ifico e Tecnol\'ogico (CNPq), of
the Ministry of Science, Technology and Innovation of Brazil}
\thanks{$^2$Partially supported by grant 2016/10009-3, São Paulo Research Foundation (FAPESP)}



\begin{abstract}
We establish the existence of analytic curves of eigenvalues for the Laplace-Neumann operator through an analytic variation of the metric of a compact Riemannian manifold $M$ with boundary by means of a new approach rather than Kato's method for unbounded operators. We obtain an expression for the derivative of the curve of eigenvalues, which is used as a device to prove that the eigenvalues of the Laplace-Neumann operator are generically simple in the space $\mathcal{M}^k$ of all $C^k$ Riemannian metrics on $M$. This implies the existence of a residual set of metrics in $\mathcal{M}^k$, which make the spectrum of the Laplace-Neumann operator simple. We also give a precise information about the complementary of this residual set, as well as about the structure of the set of the deformation of a Riemannian metric which preserves double eigenvalues.
\end{abstract}
\maketitle

\section{Introduction}
In her seminal work Uhlenbeck~\cite{Uhlenbeck} proved groundbreaking results on generic properties for eigenvalues and eigenfunctions of the Laplace-Beltrami operator $\Delta_{g}$ on a closed $n$-dimensional Riemannian manifold $(M^n,g)$. From a qualitative point of view, one of the most beautiful results in \cite{Uhlenbeck} is the celebrated Theorem~8 asserting that the set of all $C^k$~Riemannian metrics $g$ for which $\Delta_g$ has simple spectrum is residual in the separable Banach space $\mathcal{M}^k$ of all $C^{k}$~Riemannian metrics on $M^n$, for any $2\leq k <\infty$ equipped with the $C^k$ topology. Over the last four decades, similar results were obtained in various directions. We refer to \cite{ahmad,henry,Marrocos} and the references therein for background on this subject.

In line with this theme, Micheletti and Pistoia~\cite[Theorem 4.1]{Micheletti} have proposed a sufficient condition for the set of the deformations of a Riemannian metric $g$ on $M^n$, which preserve the multiplicity $m\geq 2$ of a fixed eigenvalue $\lambda(g)$ associated with $g$, to be locally a submanifold of codimension $\frac{1}{2}m(m +1)-1$ inside the Banach space $\mathcal{S}^{2,k}$ of all $C^k$~symmetric covariant $2$--tensors on $M$. They proved that such a condition is easily fulfilled when $n=2$ and $m=2$, see \cite[Theorem 4.3]{Micheletti}. Explicit examples were given, which are in accordance with their results. Shortly after, Teytel defined a notion of \emph{meager codimension} in an infinite-dimensional separable Banach space (see \cite[Section 2]{teytel}) that can be used to give a precise information about the set of metrics which the Laplacian has at least one eigenvalue with multiplicity greater than one. The crucial step in approach of Teytel has been to impose a condition, which is closely related to the strong Arnold hypothesis for double eigenvalues but significantly easier to check. We refer to \cite{teytel} for a background on the strong Arnold hypothesis.

Nevertheless very little work has been done so far to address issues about generic properties of eigenvalues and eigenfunctions of the Laplace-Beltrami operator on compact Riemannian manifolds, subjected to boundary Neumann conditions. To shorten notation, it hereinafter will refer as Laplace-Neumann operator. The single and most significant reason for this, is due to the difficulty of dealing with normal derivatives of functions in boundary when $g$ varies through $\mathcal{M}^k$. Furthermore, even with the help of important results from perturbation theory of linear operators, usually available for the Dirichlet case, the Laplace-Neumann operator requires own treatment.

Our purpose is to study some cases of generic properties for the Laplace-Neumann operator on compact manifolds. It follows that every metric $g\in\mathcal{M}^{k}$ determines a sequence
\begin{equation*}
0 = \mu_{0}(g) < \mu_1(g)\leq\mu_2(g)\leq \cdots\leq\mu_k(g)\leq \cdots
\end{equation*}
of eigenvalues of $\Delta_{g}$ counted with their multiplicities. We regard each eigenvalue $\mu_{k}(g)$ as a function of $g$ in $\mathcal{M}^{k}$.

In Section~\ref{sec existence}, we ensuring the existence of eigenvalue-curves given by a real analytic one-parameter family of Riemannian metrics $g(t)$ on $M$. As is well-known, these curves were already obtained in the Dirichlet case by Berger~\cite{Berger} using the perturbation theory for linear operators of Kato~\cite{Kato}. However, we present a new approach to prove this existence result, we believe that this technique is of independent interest. Our strategy focuses on reducing the infinite-dimensional problem to a finite-dimensional one and then use the Kato Selection theorem, as an optional device. To reduce the dimension we shall use the method of Liapunov-Schmidt, along the same lines of thinking as in~\cite{Marrocos}. After all, we will are in a position to prove our existence result, see Proposition~\ref{thmExistence}.

Thanks to Proposition~\ref{thmExistence}, it makes sense to get a formula of Hadamard type for the eigenvalues of the Laplace-Neumann operator, see Lemma~\ref{pro1}. Now, we begin by properly stating our results. Let $\lambda_0$ be an eigenvalue of $\Delta_{g_0}$ with multiplicity $m(\lambda_0)$, we recall that there are a positive number $\epsilon_{\lambda_0, g_0}$ and a neighborhood $\mathcal{V}_{\epsilon}$ in $\mathcal{M}^k$, such that for all $g\in \mathcal{V}_{\epsilon}$ one has
\begin{equation}\label{kato-continuity}
\sum_{\{|\lambda-\lambda_0|<\epsilon_{\lambda_0, g_0}\}\cap spec(\Delta_g)}m(\lambda)=m(\lambda_0).
\end{equation}
Indeed, equation~\eqref{kato-continuity} is a consequence of the continuity of a finite system of eigenvalues, see \cite[Section 5, Chapter 4]{Kato}. In this setting, we prove the following generic result.

\begin{theorem}\label{thm1}
Let $M^n$, $n\geq2$, be a compact oriented smooth manifold. Take $g_0$ in $\mathcal{M}^k$, and let $\lambda$ be an eigenvalue of the Laplace-Neumann operator $\Delta_{g_0}$ of multiplicity $m\geq2$. Take a positive number $\epsilon_{\lambda,g_0}$ and a neighborhood $\mathcal{V}_{\epsilon}$ of $g_0$ in $\mathcal{M}^k$ as in~\eqref{kato-continuity}. Then for each open neighborhood $\mathcal{U}\subset \mathcal{V}_{\epsilon}$ of $g_0$ there is $g\in\mathcal{U}$ such that all eigenvalues $\lambda(g)$ of $\Delta_g$ with $|\lambda(g)-\lambda|<\epsilon_{\lambda,g_0}$ are simple.
\end{theorem}

In what follows $\mathcal{D}$ stands for the set of all $g\in\mathcal{M}^k$ such that the eigenvalues of the Laplace-Neumann operator $\Delta_g$ are all simple. Each $g\in\mathcal{D}$ can be obtained as a generic member of a differentiable family of self-adjoint operators $A(q)$ indexed by a parameter $q\in\mathcal{M}^k$, see Section~\ref{SectionProofThm2}.

\begin{corollary}\label{CR}
The set $\mathcal{D}$ is residual in $\mathcal{M}^k$.
\end{corollary}

We obtain our results by using Hadamard type formulas for the eigenvalues (see Lemma~\ref{pro1}), which is more appropriate to the Laplace-Neumann case. Moreover, such formulas give us, without much effort, the optimal tool to apply Teytel's approach~\cite{teytel} from which we prove:

\begin{theorem} \label{thmManifold}
Let $M^n$, $n\geq2$, be a compact oriented smooth manifold.
\begin{enumerate}
\item The set $\mathcal{M}^k\backslash\mathcal{D}$ has meager codimension $2$ in $\mathcal{M}^k.$
\item Take $g_0\in\mathcal{M}^k$, and let $\lambda$ be an eigenvalue of the Laplace-Neumann operator $\Delta_{g_0}$ of multiplicity $2$. Then, in a neighborhood of $g_0$, the set of all $g\in\mathcal{M}^k$ such that $\Delta_g$ admits an eigenvalue $\lambda(g)$ of multiplicity $2$ near $\lambda$ form a submanifold of meager codimension $2$ in $\mathcal{S}^{2,k}$.
\item Consider the same set up as Item $(2)$. Then, in a neighborhood of $g_0$, the set of all $g\in\mathcal{M}^k$ which preserves double eigenvalues, i.e., $\Delta_g$ admits an eigenvalue $\lambda(g)$ of multiplicity $2$ such that $\lambda(g)=\lambda(g_0)$, form a submanifold of meager codimension $3$ in $\mathcal{S}^{2,k}$.
\end{enumerate}
\end{theorem}

\section{Preliminaries}
Let $M^n$, $n\geq2$, be a compact oriented $n$-dimensional smooth manifold with boundary $\partial M$, and let $g\in\mathcal{M}^k$. We consider the inner product $\langle T,S\rangle=\mathrm{tr}\big(TS^{*}\big)$ induced by $g$ acting on the space of $(0,2)$--tensors on $M$, where $S^*$ denotes the adjoint tensor of $S$. Clearly, in local coordinates we have
\begin{eqnarray*}\langle T,S\rangle = g^{ik}g^{jl}T_{ij}S_{kl}.
\end{eqnarray*}
Furthermore, for $f\in C^\infty(M)$ we have $\Delta_gf=\langle\nabla^2f,g\rangle$ where $\nabla^2f=\nabla df$ is the Hessian of $f$. We recall that  each $(0,2)$--tensor $T$ on $(M,g)$ can be associated to a unique $(1,1)$--tensor by $g(T(X),Y) := T(X,Y)$ for all $X,Y\in\mathfrak{X}(M)$. We shall slightly abuse notation here by writing the letter $T$ to indicate this $(1,1)$--tensor. So, we can consider the $(0,1)$--tensor given by
\begin{equation*}
(\dv T)(v)(p) = \mathrm{tr}\big(w \mapsto (\nabla_w T)(v)(p)\big)
\end{equation*}
where $p\in M$ and $v\in T_pM.$

Before deriving our main results we need a well-known lemma which will be crucial in the sequel.
\begin{lemma}\label{lem1}
Let  $T$ be a symmetric $(0,2)$--tensor on a Riemannian manifold $(M,g)$ and $\varphi$ a smooth function  on $M$. Then we have
\begin{equation*}
\dv (T(\varphi Z))= \varphi\langle \dv T,Z\rangle + \varphi\langle \nabla Z, T\rangle + T(\nabla\varphi ,Z),
\end{equation*}
for each $Z\in\mathfrak{X}(M)$, and we are considering the duality $(\dv T)(Z)=\langle \dv T, Z\rangle$.
\end{lemma}

Let $t\mapsto g(t)$ be a smooth variation of $g$ such that $(M,g(t), \dMt)$ is a Riemannian manifold, where $\dMt$ is the volume element of $g(t)$, and let $\dnt$ be the volume element of $g(t)$ restricted to $\partial M$. Denote by $H$ a $(0,2)$--tensor given by $H_{ij}=\frac{d}{dt}\big|_{t=0}g_{ij}(t)$ and
set $h=\langle H,g\rangle$. Similarly  $\tilde{h}$ stands for the trace of the $(0,2)$--tensor $\tilde{H}$ induced by the derivative of $g(t)$ restricted to $\partial M$.  It is easily seen that
\begin{equation*}
\frac{d}{dt}\dMt=\frac{1}{2}h\dM \quad \mbox{and} \quad  \frac{d}{dt}\dnt = \frac{1}{2}\tilde{h}\dn.
\end{equation*}

Since there is no danger of confusion, we will also write $\langle \,, \rangle$ to indicate the metric $g(t)$. Given $X,Y\in \mathfrak{X}(M)$ we can write $X=g^{ij}(t)x_i(t)\partial_j$ and $Y=g^{kl}(t)y_k\partial_l$, with $x_i(t)=\langle X,\partial_i\rangle$ and $y_k(t)=\langle Y,\partial_k\rangle$.

For ease of notation, let $\dot{X}:=g^{ij}x_i'\partial_j$ and $\dot{Y}:=g^{kl}y_k'\partial_l$, with $x'_i(t)=\frac{d}{dt}x_i$ and $y'_i(t)=\frac{d}{dt}y_i$. Then for every $X,Y\in \mathfrak{X}(M)$ and every $f, l\in C^\infty(M)$, the following properties can be verified:
\vskip .2cm
\begin{itemize}
\item[(P1)] \label{eq1lem2}$\frac{d}{dt}\langle X,Y\rangle =-H(X,Y)+\langle \dot{X},Y \rangle+\langle X,\dot{Y} \rangle$
\item[(P2)] \label{eq2lem2}$\frac{d}{dt}\langle \nabla_t f,\nabla_t l\rangle=-H(\nabla f,\nabla l)$
\item[(P3)] \label{eq3lem2}$\frac{d}{dt}\langle \nu_t,\nabla_t l(t)\rangle=-H(\nu,\nabla l)+\frac{1}{2}H(\nu,\nu)\langle\nu,\nabla l\rangle+\langle\nu,\nabla l'\rangle,$
\end{itemize}
\vskip .1cm
where $\nu_t=\frac{\nabla_t f}{|\nabla_t f|}$ and $\nabla_t$ means the gradient with respect to $g(t).$ Indeed,
\begin{eqnarray*}
\begin{split}
\frac{d}{dt}\langle X,Y\rangle &=\frac{d}{dt}\big(g^{ij}(t)x_i(t)y_j(t)\big)\\
&=-g^{ik}H_{kl}g^{lj}x_iy_j+g^{ij}x_i'y_j+g^{ij}x_iy_j'\\
&= -H(X,Y)+\langle \dot{X},Y \rangle+\langle X,\dot{Y} \rangle
\end{split}
\end{eqnarray*}
and this proves (P1). For (P2), observe that if $X=\nabla f$ then $x_i=\langle\nabla f,\partial_i\rangle=\partial_i f$ is independent on $t$. For (P3), it suffices to note that $\nu_i=\frac{1}{|\nabla f|}\langle\nabla f,\partial_i\rangle$ which implies $\nu_i'=\frac{1}{2}\frac{1}{|\nabla f|}H(\nu,\nu)\partial_if$ so that $\langle\dot{\nu},\nabla l\rangle=\frac{1}{2}H(\nu,\nu)\langle\nu,\nabla l\rangle.$

\section{Ingredients for the Neumann Problem}
Throughout this section $t\mapsto g(t)$ stands for a smooth variation of $g$.

\begin{lemma}\label{lem2}
If $\nu$ is the exterior normal field at $\partial M$, then
\begin{equation*}\label{dnu}
\frac{d}{dt}\Big|_{t=0}\nu(t)=-H(\nu)+\frac{1}{2}H(\nu,\nu)\nu.
\end{equation*}
\end{lemma}
\begin{proof}
Let $f$ be a smooth function on $M$ such that $\nu(t)=\frac{\nabla_t f}{|\nabla_t f|}.$ Note that
\begin{eqnarray*}
\nonumber\frac{d}{dt}\nabla_tf &=&  -g^{ik}g^{js}H(\partial_k,\partial_s)\partial_if\partial_j = -g^{js}H(\nabla_tf,\partial_s)\partial_j\\
&=& -g^{js}\langle H(\nabla_tf),\partial_s\rangle\partial_j = -H(\nabla_tf).
\end{eqnarray*}
Hence, we have
\begin{equation*}
\frac{d}{dt}\Big|_{t=0}\nu(t)= -\frac{1}{2|\nabla f|^3}\langle-H(\nabla f),\nabla f\rangle\nabla f-\frac{1}{|\nabla f|}H(\nabla f) = -H(\nu)+\frac{1}{2}H(\nu,\nu)\nu.
\end{equation*}
\end{proof}

Our next ingredient is an integral formulae, which provides a generalization of Equation~$3.3$ due to Berger~\cite{Berger}.
\begin{lemma}\label{prop3}
The following holds for any $f,l\in C^{\infty}(M).$
\begin{equation*}
\int_{M}l\Delta'f \dM =\int_M l\Big(\frac{1}{2}\langle dh,df\rangle- \langle div H,df\rangle-\langle H,\nabla^2f\rangle\Big)\dM,
\end{equation*}
where $\Delta':=\frac{d}{dt}\big|_{t=0}\Delta_{g(t)}.$
\end{lemma}

\begin{proof}
By Stokes's Theorem
\begin{eqnarray*}
\int_{M}l\Delta_{g(t)}f \dMt=-\int_{M}\langle df,dl\rangle \dMt + \int_{\partial M}l\langle\nu_t,\nabla_t f\rangle \dnt.
\end{eqnarray*}
Applying properties (P2) and (P3), at $t=0$
\begin{eqnarray*}
\int_Ml\Delta'f \dM + \frac{1}{2}\int_Mlh\Delta f \dM&=&\int_MH(\nabla f,\nabla l)\dM - \frac{1}{2}\int_Mh\langle df,dl\rangle \dM\\
&&+\int_{\partial M}l\Big(-H(\nu,\nabla f)+\frac{1}{2}H(\nu,\nu)\frac{\partial f}{\partial\nu}\Big)\dn\\
&&+\int_{\partial M}l\frac{\tilde{h}}{2}\langle\nu,\nabla f\rangle \dn.
\end{eqnarray*}
Rearranging this latter equation, we have
\begin{eqnarray}\label{eq2lem1}
\int_Ml\Delta'f \dM &=&\int_MH(\nabla f,\nabla l)\dM-\int_{\partial M}lH(\nu,\nabla f)\dn \\
\nonumber&&-\frac{1}{2}\int_M\big(h\langle df,dl\rangle+lh\Delta f\big) \dM+\frac{1}{2}\int_{\partial M}l\big(\tilde{h}+H(\nu,\nu)\big)\frac{\partial f}{\partial\nu} \dn.
\end{eqnarray}
Letting $T=H$, $\varphi=l$ and $Z=\nabla f$ in Lemma~\ref{lem1}, we get
\begin{eqnarray}\label{eq3lem1}
\int_{\partial M} lH(\nu,\nabla f)\dn&=&\int_M div(H(l\nabla f))\dM\\
\nonumber&=&\int_Ml\big(\langle div H,df\rangle+\langle H,\nabla^2f\rangle\big)\dM + \int_MH(\nabla f,\nabla l)\dM.
\end{eqnarray}
Moreover,
\begin{equation}\label{eq4lem1}
\int_{\partial M}lh\langle\nu,\nabla f\rangle \dn=\int_M\big(lh\Delta f+h\langle df,dl\rangle\big)\dM+\int_Ml\langle df,dh\rangle \dM.
\end{equation}
Inserting \eqref{eq3lem1} and \eqref{eq4lem1} into \eqref{eq2lem1}, we obtain
\begin{eqnarray*}
\int_{M}l\Delta'f \dM &=&\int_M l\Big(\frac{1}{2}\langle dh,df\rangle- \langle div H,df\rangle-\langle H,\nabla^2f\rangle\Big)\dM \\
&& + \frac{1}{2}\int_{\partial M}l\big(-h + H(\nu,\nu)+\tilde{h}\big)\frac{\partial f}{\partial\nu} \dn,
\end{eqnarray*}
for all $l\in C^{\infty}(M)$, which is sufficient to complete the proof of the lemma, since $\tilde{h}=\mathrm{tr}_g(H|_{\partial M})=h - H(\nu,\nu)$.
\end{proof}

In what follows, we obtain Hadamard type formulas for the eigenvalues of the Laplace-Neumann operator.

\begin{lemma}\label{pro1}
Let  $\{\phi_{i}(t)\}\subset C^{\infty}(M)$ be a differentiable family of functions
and $\lambda(t)$ a differentiable family of real numbers such that $\langle\phi_{i}(t),\phi_{j}(t)\rangle_{L^2(M,\dMt)}=\delta_i^j$  for all $t$ and
$$\left\{
  \begin{array}{ccccccc}
    -\Delta_{g(t)}\phi_{i}(t) &=& \lambda(t)\phi_{i}(t) & \hbox{in} \,\,\;M\\
    \frac{\partial}{\partial\nu_t}\phi_{i}(t)&=&0 & \hbox{on}\,\,\; \partial M.
  \end{array}
\right.$$
Then,
\begin{equation}
\lambda'(0)\delta_i^j=\int_{M}\langle\frac{1}{4}\Delta(\phi_i\phi_j)g-d\phi_i\otimes d\phi_j, H\rangle \dM.
\end{equation}
\end{lemma}

\begin{proof}
Taking the derivative with respect to $t$ at $t=0$ in both sides of the equation $-\Delta_{g(t)} \phi_{i}(t)=\lambda(t)\phi_{i}(t)$, we have $-\Delta'\phi_i-\Delta\phi'_i=\lambda'\phi_i+\lambda\phi'_i$. Thus
\begin{equation*}
-\int_{M}(\phi_j\Delta'\phi_i+\phi_j\Delta\phi'_i)\dM
=\int_{M}(\lambda'\phi_j\phi_i-\phi_i'\Delta\phi_j)\dM.
\end{equation*}
Since $\langle\nu_t,\nabla_t\phi_{i}(t) \rangle=\frac{\partial}{\partial\nu_t}\phi_{i}(t)=0$ on $\partial M$, we deduce that
\begin{eqnarray*}\langle\nu,\nabla\phi'_i\rangle=H(\nu,\nabla\phi_i)-\frac{1}{2}H(\nu,\nu)\langle\nu,\nabla\phi_i\rangle=H(\nu,\nabla\phi_i)\quad \text{at } t=0.
\end{eqnarray*}
Integration by parts gives
\begin{eqnarray*}
\lambda'\delta_i^j &=& - \int_{M}\phi_j\Delta'\phi_i \dM - \int_{\partial M}\phi_j\frac{\partial}{\partial\nu}\phi'_i\dn\\
 &=& - \int_{M}\phi_j\Delta'\phi_i \dM - \int_{\partial M}\phi_jH(\nu,\nabla\phi_i)\dn.
\end{eqnarray*}
Whence,
\begin{eqnarray*}
-2\lambda'\delta_i^j &=& \int_{M}\phi_j\Delta'\phi_i \dM + \int_{M}\phi_i\Delta'\phi_j \dM + \int_{\partial M}\phi_iH(\nu,\nabla\phi_j)\dn \\
&&+ \int_{\partial M}\phi_jH(\nu,\nabla\phi_i)\dn\\
&=&\int_{M}\langle \frac{1}{2}dh-div H,\phi_jd\phi_i + \phi_id\phi_j \rangle \dM - \int_{M}\langle H,\phi_j\nabla^2\phi_i + \phi_i\nabla^2\phi_j\rangle \dM\\
&&+ \int_{\partial M}\phi_iH(\nu,\nabla\phi_j)\dn + \int_{\partial M}\phi_jH(\nu,\nabla\phi_i)\dn\\
&=&\int_{M}\langle \frac{1}{2}dh ,d(\phi_i\phi_j)\rangle \dM - \int_{M}\phi_j\big(\langle div H, d\phi_i\rangle+\langle H,\nabla^2\phi_i\rangle\big) \dM\\
&&+\int_{\partial M}\phi_jH(\nu,\nabla\phi_i)\dn -\int_{M}\phi_i\big(\langle div H,d\phi_j\rangle+\langle H,\nabla^2\phi_j\rangle\big) \dM\\
&&+ \int_{\partial M}\phi_iH(\nu,\nabla\phi_j)\dn.
\end{eqnarray*}
Next use divergence's theorem together with Lemma~\ref{lem1} to deduce
\begin{eqnarray*}
-2\lambda'\delta_i^j &=& - \int_{M}\frac{h}{2}\Delta(\phi_i\phi_j)\dM + 2 \int_{M} H(\nabla\phi_i,\nabla\phi_j) \dM
\end{eqnarray*}
or equivalently
\begin{equation*}
\lambda'\delta_i^j = \int_{M}\big\langle\frac{1}{4}\Delta(\phi_i\phi_j)g - d\phi_i\otimes d\phi_j, H\big\rangle \dM.
\end{equation*}
\end{proof}

\section{The existence result}\label{sec existence}
We now describe the details of the proof of Proposition~\ref{thmExistence}. Here, we use the Liapunov-Schmidt method, along the same lines as done by Henry~\cite{henry}, and continued by Marrocos and Pereira~\cite{Marrocos}.

Specifically, we consider the Neumann-problem:
\begin{equation}\label{plrp}
\left\{\begin{array}{lccc}
(\Delta_t+ \lambda)u = 0& in \ \ M\\
\frac{\partial u}{\partial \nu_t} =0 & \ \ on\ \ \partial M,
\end{array}\right.
\end{equation}
where $(M,g_0)$ is an orientable, compact $n$-dimensional Riemannian manifold with boundary $\partial M$, $\Delta_t:=\Delta_{g(t)}$, $t\mapsto g(t)$ is an analytic variation of $g_0$, with $g(0)=g_0$, and $\nu_t$ is a one-parameter family of unit exterior vectors along with $(\partial M,g(t))$.

\begin{lemma}\label{tcaln}
Let $\lambda_0$ be an eigenvalue of the Laplace-Neumann operator of multiplicity $m\geq2$. Then for every $\epsilon>0$ there is $\delta>0$ so that for each $|t|<\delta$, there exist exactly $m$ eigenvalues (computing their multiplicities) to the problem {\it\eqref{plrp}} in the interval $(\lambda_0-\epsilon,\lambda_0+\epsilon)$.
\end{lemma}

\begin{proof}
Let $\{\phi_j\}_{j=1}^m$ be an orthonormal basis associated with $\lambda_0$, and let
\begin{equation*}
Pu=\sum_{j=1}^m\phi_j\int_M\phi_j u \dg
\end{equation*}
be the orthonormal projection on the corresponding eigenspace. As it is well-known, $P$ induces a splitting $L^2(M,\dM_{0})={\mathcal{R}}(P)\oplus\mathcal{N}(P)$ so that any function $u$ in $L^2(M,\dM_0)$ can be written as $u=\phi+\psi$, where  $\phi\in\mathcal{R}(P)=ker(\Delta+\lambda_0)$ and $\psi\in \mathcal{N}(P)$.

With this in mind, the Neumann-problem can be equivalently viewed as a system of equations, as to the system:
\begin{equation}\label{deqlr}
\left\{\begin{array}{ccccccc}
(I-P)(\Delta_t + \lambda)(\phi+\psi)&=0&      &\text{in }  M\\
P(\Delta_t+ \lambda)(\phi+\psi)&= 0&           &\text{in }  M\\
\frac{\partial }{\partial \nu_t}(\phi+\psi)&=0&    &\,\,\,\,\,\,\text{on } \partial M.
\end{array}\right.
\end{equation}

To solve it, we first observe that since $\phi_j$ and $\psi$ are orthonormal, by the divergence theorem we must have
\begin{equation*}
P(\Delta + \lambda)\psi  =  \sum_{j=1}^m\phi_j\int_{M}\phi_j(\Delta + \lambda)\psi\dg
 = \sum_{j=1}^m\phi_j\int_{\partial M}\phi_j \frac{\partial\psi}{\partial \nu} \db
\end{equation*}
which implies
\begin{equation*}
(\Delta+\lambda)\psi  =  (I-P)\big((\Delta+\lambda)\psi\big)+\displaystyle\sum_{j=1}^m\phi_j\int_{\partial M}\phi_j \frac{\partial\psi}{\partial \nu} \db.
\end{equation*}
Thus, we get
\begin{equation*}
(\Delta+ \lambda)\psi+ (I-P)(\Delta_t - \Delta)(\phi+\psi)-\sum_{j=1}^m\phi_j\int_{\partial M}\phi_j\frac{\partial \psi}{\partial \nu}\db =0.
\end{equation*}
Moreover, the part concerning the boundary in \eqref{deqlr} can be rewritten as
\begin{equation*}
\frac{\partial \psi }{\partial \nu} +\left(\frac{\partial }{\partial \nu_t} - \frac{\partial}{\partial \nu} \right)(\phi+\psi)=0.
\end{equation*}
Hence solving the first and third equations of \eqref{deqlr}, is equivalent to finding the zeros of the application
\begin{eqnarray*}
F :  \mathbb{R}\times\mathbb{R}\times\mathcal{R}(P)\times H^2(M)\cap\mathcal{N}(P)&\longrightarrow &\mathcal{N}(P)\times H^{\frac{3}{2}}(M)\\
(t,\lambda,\phi,\psi)&\mapsto&\big(F_1(t,\lambda,\phi,\psi),F_2(t,\lambda,\phi,\psi)\big),
\end{eqnarray*}
where
\begin{equation*}
\left\{\begin{array}{lcc}
F_1=(\Delta + \lambda)\psi+ (I-P)(\Delta_t - \Delta)(\phi+\psi)-\displaystyle\sum_{j=1}^m\phi_j\int_{\partial M}\phi_j\frac{\partial\psi}{\partial \nu}\db\\
F_2 = \frac{\partial \psi}{\partial \nu} + \big(\frac{\partial }{\partial \nu_t} - \frac{\partial}{\partial \nu}\big)(\phi+\psi).
\end{array}\right.
\end{equation*}

Note that $F$ depends differentially on the variables $\lambda$, $t$, $\psi$ e $\phi$. Our intent is to use the Implicit Function Theorem to show that $F(t,\lambda,\phi,\psi)=(0,0)$ admits a solution $\psi$ as function of $\lambda$, $t$ and $\phi$. To this end, we observe that if $t = 0, \lambda=\lambda_0,\psi=0$ then
\begin{equation}\label{TAI}
\frac{\partial F}{\partial \psi}(0,\lambda_0 , 0 , 0 ) \dot{\psi} = \Big((\Delta + \lambda_0) \dot{\psi} - \sum_{j=1}^m \phi_j \int_{\partial M}\phi_j\frac{\partial \dot{\psi}}{\partial \nu}\db \, ,\, \frac{\partial \dot{\psi} }{\partial \nu}\Big).
\end{equation}
We claim now that the map given in \eqref{TAI} is an isomorphism from $H^2(M)\cap\mathcal{N}(P)$ onto $\mathcal{N}(P)\times H^{\frac{3}{2}}(M)$. Indeed, the proof of this fact can be found in~\cite{lion}.

Hence, by Implicit Function Theorem there exist positive numbers $\delta$, $\epsilon$ and a function $S(t , \lambda)\phi$ of class $C^1$ at the variables $(t , \lambda)$ such that for every $|t| < \delta$ and  $\lambda \in (\lambda_0-\epsilon, \lambda_0+\epsilon)$, $F(t ,\lambda,\phi, S(t,\lambda)\phi)=(0,0)$. Furthermore, $S(t , \lambda)\phi$ is analytic at $\lambda$ and linear at $\phi$. This solves equation~\eqref{deqlr} in relation to $\psi$.

We now observe that for every  $\phi\in\mathcal{R}(P)$ there exist real numbers $c_1,c_2,\dots,c_m$ so that $\phi = \sum_{j=1}^{m}c_j\phi_j$. Thus, the second equation in~\eqref{deqlr} can be equivalently seen as a system of equations on the variables  $c_1,\dots,c_m$ as below
\begin{equation*}
\sum_{j=1}^m c_j\int_{M}\phi_k (\Delta_t + \lambda)(\phi_j+S(t,\lambda)\phi_j)\dg=0,\quad k=1,2,\dots, m.
\end{equation*}
In this way, $\lambda$ is an eigenvalue of $\Delta_t$ if and only if $\det A(t,\lambda)=0$, where $A(t,\lambda)$ is given by
\begin{equation*}
A_{kj}(t,\lambda)=\int_{M}\phi_k(\Delta_t + \lambda) (\phi_j+S(t,\lambda)\phi_j)\dg.
\end{equation*}
Furthermore, the associated eigenfunctions are given by
\begin{equation*}
u(t, \lambda)=\sum_{j=1}^{m}c_j(\phi_j+S(t,\lambda)\phi_j).
\end{equation*}
In other words, $c=(c_1,\ldots,c_m)$ must to satisfy $A(t,\lambda) c=0$. It turns out that by Rouch\'e Theorem, we have that: For every $\epsilon>0$ there is $\delta > 0$ so that if $|t - t_0| < \delta $, then there exist exactly $m$-roots of $\det A(t ,\lambda)=0$ in $(\lambda_0-\epsilon,\lambda_0+\epsilon)$.
\end{proof}

It is worth mentioning that the above proof still does not ensure the existence of an analytic curve of eigenvalues for (\ref{plrp}). However, the next result goes in this direction.

\begin{proposition}\label{thmExistence}
Let $M^n$, $n\geq2$, be a compact oriented smooth manifold, and let $g(t)$ be a real analytic one-parameter family of Riemannian metrics on $M$ with $g(0)=g_0$. Assume $\lambda$ is an eigenvalue of multiplicity $m$ for the Laplace-Neumann operator $\Delta_{g_0}$. Then there exist $\varepsilon>0$ and $t$-analytic functions $\lambda_{i}(t)$ and $\phi_{i}(t)$, ($i=1,\dots,m$) such that $\langle\phi_{i}(t),\phi_{j}(t)\rangle_{L^2(M,\dMt)}=\delta_i^j$ and the following relations hold for every $|t|<\varepsilon$:
\begin{enumerate}
\item[(i)] $\Delta_{g(t)}\phi_{i}(t)=\lambda_{i}(t)\phi_{i}(t)$ in $M$,
\item[(ii)] $\frac{\partial}{\partial\nu_t}\phi_i(t)=0$ on $\partial M$,
\item[(iii)] $\lambda_{i}(0)=\lambda$.
\end{enumerate}
\end{proposition}

\begin{proof}
Assume the same conditions of Lemma~\ref{tcaln}. We must show that there exist $m$-analytic curves of eigenvalues $\lambda_j(t)$ for \eqref{plrp} associated to $m$-analytic curves eigenfunctions  $\phi_j(t)$.

The proof strategy is reducing the problem to a finite-dimension analogous one and applying Kato's Selection theorem, see~\cite{Kato}. For this, we will make a slightly different construction than that made in previous proposition.

In order to overcome this drawback, we slightly will change the previous proof of such a way that the new obtained matrix will come to be symmetric.

Let $\{\phi_j\}_{j=1}^m$ be orthonormal eigenfunctions of the Laplace-Neumann associated with $\lambda_0$. For each $j=1,\ldots, m$ consider the following problem:
\begin{equation}\label{pfp}
\left\{\begin{array}{ccccccc}
(\Delta +\lambda_0)u&=0& &\text{in }  M\\
\frac{\partial }{\partial \nu_t}(\phi_j+ u)&=0&   &\,\,\,\,\text{on } \partial M\\
Pu=\displaystyle\sum_{j=1}^m\phi_j\int_{M}\phi_ju\dg&=0& &\,\,\text{in } M.
\end{array}\right.
\end{equation}

Consider now the orthogonal complement $[\phi_j]^{\bot}$ of $ker(\Delta+\lambda_0)$ in $L^2(M,\dM_0)$ and define
\begin{equation*}
F: (-\delta, \delta)\times H^2(M,\dM_0)\longrightarrow [\phi_j]^{\bot}\times\mathcal{R}(P)\times H^{\frac{3}{2}}(M,\dM_0)
\end{equation*}
by
\begin{equation*}
F(t,w)=\big((\Delta+\lambda_0)w,\, Pw,\, \frac{\partial }{\partial \nu_t} (\phi_j+w)\big).
\end{equation*}
Exactly as before we get that $\frac{\partial F}{\partial w}(0,0)$ is an isomorphism, so by Implicit Function Theorem there exist $\delta > 0$ and an analytic function $w_j(t)$ defined on $|t - t_0| < \delta$ such that $F(t,w_j(t)) = 0$. In addition, we obtain for each $|t - t_0| < \delta$ a linearly independent set of functions $\{\varphi_j(t)\}_{j=1}^m$, given by $\varphi_j(t)=\phi_j+w_j(t)$,
that satisfy equation~\eqref{pfp}. By using the Gram-Schmidt orthonormalization process with respect to the inner product
\begin{equation*}
(u,v):=\int_{M}uv\, \dMt,
\end{equation*}
we can without loss of generality assume  that $\{\varphi_j(t)\}_{j=1}^m$ is biorthogonal. Note that the functions $\varphi_j(t)$ belong to $D_t=\{u \in H^{2}(M,\dM_0), \frac{\partial u}{\partial \nu_{t}} = 0\}$. Moreover, since $\Delta_t$ is selfadjoint with respect to the inner product defined above, it follows that the matrix $\int_{M}\varphi_j \Delta_t \varphi_k \dMt$ is symmetric.

For a given $T\in \mathcal{S}^{2,k}$, we define a family of Riemannian metric on $M$ by $g(t) = g_0 +tT$, and let $P(t)$ be given by
\begin{equation*}
P(t)u=\sum_{j=1}^m\varphi_j(t)\int_{M}u\varphi_j(t)\dMt.
\end{equation*}
We finally define for each $j=1,\dots, m$,
\begin{eqnarray*}
 G_j :  (-\epsilon, \epsilon)\times\mathbb{R}\times H^2(M)&\longrightarrow & [\phi_j]^{\bot}\times H^{\frac{3}{2}}(M)\times \mathcal{R}(P)\\
 (t,\lambda,w)&\mapsto & \big(G_{j1}(t,\lambda,w),G_{j2}(t,\lambda,w),G_{j3}(t,\lambda,w)\big)
\end{eqnarray*}
by
\begin{equation*}
\left\{\begin{array}{lcc}
G_{j1}=(I-P(t))((\Delta_t + \lambda))(w+\varphi_j(t))\\
G_{j2}=\frac{\partial }{\partial \nu_t}w;\\
G_{j3}= P(t)w.\end{array}\right.
\end{equation*}
Once again, Implicit Function Theorem provide a number   $\delta>0$ and functions $w_j(t,\lambda)$ such that for any $|t- t_0|<\delta$ and every $|\lambda-\lambda_0|<\delta$, the equality $G_j(t,\lambda,w_j(t,\lambda))=(0,0,0)$ holds. As we know,  $\lambda$  is an eigenvalue for \eqref{plrp} iff there exists a nonzero $m$-tuple $c=(c_1,\ldots,c_m)$ of real numbers such that $A(t,\lambda)c=0$, where
\begin{equation}
A_{ij}(t,\lambda)= \int_{M}\varphi_i(t)(\Delta_t+\lambda)(\varphi_j(t)+w_j(t,\lambda))\dMt.
\end{equation}
That is, $\lambda$ is an eigenvalue of \eqref{plrp} iff $\det A(t,\lambda)=0$. By Rouch\'e's Theorem, there exist $m$ roots near $\lambda_0$ counting multiplicity, for each $t$. So, by Puiseux's Theorem \cite{wall} there exist $m$-analytic functions $t\to \lambda_i(t)$ which locally solve the equation $\det A(t,\lambda)=0$. It can be easily seen that  $A$ is symmetric and hence, by Kato's Selection theorem \cite{Kato}, we can find an analytic curve $c^i(t)\in\mathbb{R}^m$ such that $A(t,\lambda_i(t))c^i(t)=0$, for each $i=1,\dots,m$. Thus $\psi_i(t)=\sum_{j=1}^m c^i_j(t)(\varphi_j+\omega_j(t,\lambda_i(t)))$ is an analytic curve of eigenfunctions for \eqref{plrp} associated with $\lambda_i(t)$. Now, reasoning exactly as Kato in \cite[p. 98]{Kato} we can obtain $m$-analytic curves of eigenvalues $\{\phi_i(t)\}_{i=1}^m$ such that $\int_{M}\phi_i(t)\phi_j(t)\dMt=\delta_i^j$.
\end{proof}

\begin{remark}
In the special case $m=m(\lambda_0)=1$, the existence of a differentiable curve of eigenvalues through $\lambda_0$ follows from Implicit Function Theorem applied to the map $F:S^k\times H^2(M,\dM_0)\times \mathbb{R}\rightarrow L^2(M, \dM_0)\times \mathbb{R}$ defined by
\begin{equation*}
F(g, u , \lambda)= ((\Delta_g + \lambda)u,\int_{M}u^2\dg).
\end{equation*}
The corresponding formulae to the derivative $\lambda'(t)$ can be obtained by letting $i=j=1$ in Lemma~\ref{pro1}.
\end{remark}

\subsection{Proof of Theorem~\ref{thm1}}
\begin{proof} We argue by contradiction. Suppose that there is an open neighborhood $\mathcal{U}\subset \mathcal{V}_\epsilon$ of $g_0$ such that for all $g\in \mathcal{U}$ the eigenvalue $\lambda(g)$ of $\Delta_{g}$ with $|\lambda(g)-\lambda|< \epsilon_{\lambda,g_0}$ has multiplicity $m\geq2$. In this case, for any symmetric $(0,2)$--tensor $T$ on $M$ the perturbation $g(t)=g_0+tT$ fails to split the eigenvalue $\lambda$. In this case, by Lemma~\ref{tcaln} $\lambda(t)$ is only eigenvalues $\epsilon$-close to $\lambda$. By Lemma~\ref{pro1}
\begin{equation}
\lambda'\delta_i^j = \int_{M}\big\langle\frac{1}{4}\Delta(\phi_i\phi_j)g_0 - d\phi_i\otimes d\phi_j, T\big\rangle \dg.
\end{equation}
Now consider the symmetrization tensor
\begin{equation}\label{Symm}
S_{ij}=\frac{1}{2}(d\phi_i\otimes d\phi_j + d\phi_j\otimes d\phi_i)
\end{equation}
and use the fact that $\langle d\phi_i\otimes d\phi_j, T\rangle=\langle d\phi_j\otimes d\phi_i,T \rangle$ to deduce
\begin{equation}
\lambda'\delta_i^j = \int_{M}\big\langle\frac{1}{4}\Delta(\phi_i\phi_j)g_0 - S_{ij}, T\big\rangle \dg.
\end{equation}
If $i\neq j$, we have
\begin{equation*}
\int_{M}\big\langle\frac{1}{4}\Delta(\phi_i\phi_j)g_0 - S, T\big\rangle \dg=0,
\end{equation*}
for all $T\in\mathcal{S}^{2,k}$. Thus
\begin{equation}\label{eq S}
\frac{1}{4}\Delta(\phi_{i}\phi_{j})g_0= S_{ij}.
\end{equation}
By taking the trace in~\eqref{eq S}, we obtain
\begin{eqnarray*}
\langle\nabla\phi_{i},\nabla\phi_{j}\rangle &=& \frac{n}{4}\Delta(\phi_{i}\phi_{j}) = \frac{n}{4}\big(\phi_{i}\Delta\phi_{j}+\phi_{j}\Delta\phi_{i}+2\langle\nabla\phi_{i},\nabla\phi_{j}\rangle\big)\\
&=&\frac{n}{2}(-\lambda\phi_i\phi_j + \langle\nabla\phi_i,\nabla\phi_j\rangle).
\end{eqnarray*}
Thus, for $n=2,$ it follows from the unique continuation principle \cite{hormander} that at least one eigenfunction would be vanish, which is a contradiction. Thus $\lambda(t)$ is simple.
For $n\geq3$ we can write
\begin{equation}\frac{n\lambda}{n-2}\phi_{i}\phi_{j}=\langle\nabla\phi_{i},\nabla\phi_{j}\rangle.
\end{equation}
By reasoning exactly as Uhlenbeck~\cite{Uhlenbeck}, for each  fixed $p\in M$ we consider $\alpha$ the integral curve in $M$ such that $\alpha(0)=p$ and $\alpha'(s)=\nabla\phi_i(\alpha(s))$. Define $\beta(s):=\phi_j(\alpha(s))$, then
\begin{eqnarray*}
\beta'(s)&=&\langle \nabla\phi_j(\alpha(s)),\alpha'(s)\rangle =\langle \nabla\phi_j,\nabla\phi_i \rangle(\alpha(s))\\
&=&\frac{n\lambda}{n-2}\phi_{i}\phi_{j}(\alpha(s)) = \frac{n\lambda}{n-2}\phi_{i}(\alpha(s))\beta(s)
\end{eqnarray*}
which is a contradiction since $M$ is compact, thereby proving the theorem.
\end{proof}
\subsection{Proof of Corollary~\ref{CR}}
\begin{proof}
Recall that a residual set is a countable intersection of open dense sets. We note that $\mathcal{D}= \cap_{j}\mathcal{D}_j$, where
\begin{equation*}
\mathcal{D}_j = \{ g\in\mathcal{M}^k : \mbox{all eigenvalues}\, \lambda \leq j  \,\,\mbox{of}\, \, \Delta_g  \,\, \mbox{are simple}\}.
\end{equation*}
Hence we need  to prove that each $\mathcal{D}_j$ are open and dense. The openness is quite straightforward, and it follows directly from Lemma~\ref{tcaln}, while the denseness is a consequence of Theorem~\ref{thm1}.
\end{proof}

\section{Proof of Theorem~\ref{thmManifold}}\label{SectionProofThm2}
In this section, we consider the weak form of the Laplace-Neumann operator\footnote{This will be crucial to apply the Teytel approach's since the domain of the Laplace Neumann operator $\Delta_g$ changes when we perturb the metric $g$. }. It can be defined as  the unique self-adjoint operator $\Delta_g$ on $H^1(M,\dM)$ whose associated bilinear form is $\alpha(u,v)=\int_{M}g(\nabla u, \nabla v)\dM$. In this case, we have
\begin{equation*}
\int_{M}u\Delta_g v \dM= \alpha(u,v).
\end{equation*}

The proof of Theorem~\ref{thmManifold} will be based on the following result due to Teytel~\cite{teytel}. \emph{Let $A(q)$ be a differentiable family of self-adjoint operators on a real Hilbert space $H$, indexed by a parameter q that belongs to a separable Banach manifold $\mathcal{X}$. Assume that the spectrum of each operator $A(q)$ is discrete, of finite multiplicity, and with no finite accumulation points. Assume also that the family $A(q)$ satisfies SAH2\footnote{See Remark~\ref{SAH2} for the definition of the condition SAH2.}. Then the set of all $q$ such that $A(q)$ has a repeated eigenvalue has meager codimension $2$ in $\mathcal{X}$.}

We need to choose a family $A(g)$ of self-adjoint operators with respect to a fixed inner product such that the Laplace-Neumann operator $\Delta_g$ is unitary equivalent to $A(g)$ for each $g\in \mathcal{M}^k$. For this, let us define the isometry operator $\mathcal{P}: L^2(M,\dg)\rightarrow L^2(M,\dM)$ given by
\begin{equation}
\mathcal{P}(u)= \frac{\sqrt[4]{|g_0|}}{\sqrt[4]{|g|}}u,
\end{equation}
where $|g|=\det(g_{ij})$. Then $\int_{M}\mathcal{P}u\mathcal{P}v\dM=\int_{M}uv\dg$, and thus the operator $A(g):= \mathcal{P}^{-1}\Delta_g\mathcal{P}$ has the same eigenvalues of $\Delta_g$. Moreover, it is self-adjoint with respect to a fixed inner product. Indeed, by density we can consider $u,v\in C^2(M)$ with boundary condition $\frac{\partial u}{\partial \nu}=\frac{\partial v}{\partial \nu}=0$, so that
\begin{eqnarray*}
\int_{M}vA(g)u\dg &\stackrel{(isom.)}{=}&\int_{M}\mathcal{P}v\Delta_g\mathcal{P}u\dM=\int_{M}\mathcal{P}u\Delta_g\mathcal{P}v\dM\\
&\stackrel{(isom.)}{=}&\int_{M}\mathcal{P}^{-1}\mathcal{P}u\mathcal{P}^{-1}\Delta_g\mathcal{P}v\dM=\int_{M}uA(g)v\dg.
\end{eqnarray*}
Furthermore, we choose $t$ small enough that $g(t)=g_0+ tT\in \mathcal{M}^k$, for a symmetric $(0,2)$--tensor $T$. So, we get
\begin{equation*}
A(g_0+tT)=A(g_0)+tA^{(1)}(g_0,T)+o(t),
\end{equation*}
where $A^{(1)}(g_0,T)=-\mathcal{P}'\Delta_{g_0}+\Delta_{g_0}\mathcal{P}'+\Delta'$ and $\mathcal{P}'=\frac{d}{dt}\big|_{t=0}\frac{\sqrt[4]{|g_0|}}{\sqrt[4]{|g_0+tT|}}.$
To complete Teytel's approach, let $\lambda$ be an eigenvalue of $A(g)$ of multiplicity $m\geq2$, and let $\{\phi_i\}_{i=1}^m$ be the corresponding family of normalized
real-valued eigenfunctions. We define the functionals
\begin{equation}\label{fij}
f_{ij}(T) = \int_{M}\phi_i A^{(1)}(g_0,T)\phi_j\dg.
\end{equation}.
\begin{remark}\label{SAH2}
The family $A(g)$ satisfies $\mathrm{SAH2}$ if for any eigenvalue $\lambda$ of $A(g)$ of multiplicity $m\geq 2$ there exist two orthonormal eigenfunctions $\phi_1$ and $\phi_2$ associated to $\lambda$ such that the functionals $f_{11}-f_{22}$ and $f_{12}$ are linearly independent. This condition is equivalent that one given by Teytel~\cite{teytel} in an abstract setting.
\end{remark}

Now we are in a position to prove our next result.

\subsection{Proof of Theorem~\ref{thmManifold}}
\begin{proof}
For each $g\in\mathcal{M}^k$ the modified operators $A(g)$ is self-adjoint, and it has discrete spectrum, of finite multiplicity, and with no finite accumulation points.
In order to prove Items $(1)$ and $(2)$ it is enough to verify that the condition of Remark~\ref{SAH2} holds for~\eqref{fij}. However, we will prove a more stronger condition, namely $f_{11}, f_{12}, f_{22}$ are linearly independent. In this way, we also can prove Item $(3)$. We begin by noticing that
\begin{equation*}
\int_M A^{(1)}(g,T) \dg=\int_M \phi_i\Delta'\phi_j\dg.
\end{equation*}
Hence, functionals~\eqref{fij} are equivalent to
\begin{equation*}
f_{ij}(T) =\int_{M}\big\langle\frac{1}{4}\Delta(\phi_i\phi_j)g_0 - S_{ij}, T\big\rangle \dg,
\end{equation*}
where  $S_{ij}$ is given by \eqref{Symm}. Next, we note that $\alpha f_{11}+\beta f_{12} +\gamma f_{22}=0$ implies
\begin{equation}\label{linear combination}
\frac{1}{4}(\alpha\Delta\phi_1^2+\beta\Delta(\phi_1\phi_2)+\gamma\Delta\phi_2^2)g_0-(\alpha S_{11}+\beta S_{12}+\gamma S_{22})=0.
\end{equation}
By taking the trace in~\eqref{linear combination}, we obtain
\begin{equation*}
\alpha((n-2)|\nabla\phi_{1}|^2 - n\lambda\phi_{1}^2)+ \beta((n-2)\langle\nabla\phi_{1},\nabla\phi_{2}\rangle - n\lambda\phi_{1}\phi_{2})+
\gamma((n-2)|\nabla\phi_{2}|^2 - n\lambda\phi_{2}^2)=0
\end{equation*}
or equivalently
\begin{equation}\label{quadratic form}
\alpha B(\phi_1,\phi_1)+\beta B(\phi_1,\phi_2)+\gamma B(\phi_2,\phi_2)=0,
\end{equation}
where
\begin{equation*}
B(\phi_i,\phi_j)=\phi_i\phi_j \,\, \mbox{if}\,\, n=2,\,\,\,\mbox{and}\,\,\,\, B(\phi_i,\phi_j) = \langle\nabla\phi_{i},\nabla\phi_{j}\rangle - \frac{n\lambda}{n-2}\phi_{i}\phi_{j}\,\,\, \mbox{if}\,\, n>2.
\end{equation*}
Suppose that the functionals $f_{11},f_{12},f_{22}$ are linearly dependent. Hence, from equation~\eqref{quadratic form} we can obtain the linear combinations $\varphi=a\phi_1 +b \phi_2$ and $\psi= c\phi_1 + d\phi_2$, for some constants $a,b,c,d$, such that $\varphi$ and $\psi$ are linearly independent, and equation~\eqref{quadratic form} becomes
\begin{equation*}
c_1B(\varphi,\varphi)+c_2B(\psi,\psi)=0, \quad \mbox{where}\quad c_i=\pm 1.
\end{equation*}

For $n=2$ case, we have $\varphi^2=\pm\psi^2$ which implies $\varphi=\pm\psi$ in some open set in $M$. It follows from the unique continuation principle (see \cite{hormander}) that $\varphi=\pm\psi$ in $M$, which is a contradiction.

For $n> 2$ case, firstly we suppose that
\begin{equation*}
0= B(\varphi,\varphi)-B(\psi,\psi) = |\nabla\varphi|^2 - \frac{n\lambda}{n-2}\varphi^2-|\nabla\psi|^2 + \frac{n\lambda}{n-2}\psi^2.
\end{equation*}
By making $\eta=\varphi+\psi$ and $\xi=\varphi-\psi$, we have $\langle\nabla\eta,\nabla\xi\rangle - \frac{n\lambda}{n-2}\eta\xi=0$. Thus, we can argue as Theorem~\ref{thm1} to get $\eta=0$, which implies that $\varphi=-\psi$ in $M$, this is a contradiction. Next, we suppose that
\begin{equation*}
0=B(\varphi,\varphi)+B(\psi,\psi) = |\nabla\varphi|^2+|\nabla\psi|^2 - \frac{n\lambda}{n-2}(\varphi^2+\psi^2).
\end{equation*}
Since $\varphi\Delta\varphi=-\lambda\varphi^2$ and $\psi\Delta\psi=-\lambda\psi^2$, then integration by parts gives
\begin{equation*}
\frac{-2\lambda}{n-2}\int_{M}(\varphi^2+\psi^2)\dg = 0.
\end{equation*}
Once again, we obtain a contradiction. This complete the proof of Item $(1)$ of the theorem.

For Item $(2)$, we observe that any eigenvalue $\lambda$ of $\Delta_{g_0}$ of multiplicity $m=2$ must satisfy the strong Arnold hypothesis (SAH) relative to the family $A(g)$. Indeed, in this case, condition SAH is equivalent to SAH2 proved in Item $(1)$. Thereby, we can apply Theorem~1.1 in~\cite{teytel} to obtain the second part of the theorem.

For Item $(3)$, since we already proved that $f_{11},f_{12},f_{22}$ are linearly independent we can apply the Implicit Function Theorem as in the proof of Theorem $1.1$ in~\cite{teytel} to get required result.
\end{proof}

\begin{remark}
The proof of Theorem~\ref{thmManifold} also proves Corollary~\ref{CR} by an application of Theorem~C in~\cite{teytel}. Moreover, this technique does not work for the space of $C^\infty$~Riemannian metrics, since we need $\mathcal{M}^k$ be a separable Banach space in order to apply Teytel's approach.
\end{remark}

\section{Concluding Remarks}
Taking into account the existence of analogous ingredients for the $\eta$-Laplacian $L_g=\Delta_g-g(\nabla\eta,\nabla\cdot)$ with Dirichlet boundary condition (see \cite{GMR}), we can proceed as in the proof of Theorem~\ref{thmManifold} to prove the next result. In what follows $\Gamma$ stands for the set of all metrics $g\in\mathcal{M}^k$ such that the eigenvalues of $L_g$ are all simple.

\begin{theorem}
Let $M^n$, $n\geq2$, be a compact oriented smooth manifold:
\begin{enumerate}
  \item The set $\mathcal{M}^k\backslash\Gamma$ has meager codimension $2$ in $\mathcal{M}^k.$
  \item Take $g_0\in\mathcal{M}^k$, and let $\lambda$ be an eigenvalue of the $\eta$-Laplacian $L_{g_0}$ of multiplicity $2$. Then, in a neighborhood of $g_0$, the set of all $g\in\mathcal{M}^k$ such that $L_g$ admits an eigenvalue $\lambda(g)$ of multiplicity $2$ near $\lambda$ form a submanifold of meager codimension $2$ in $\mathcal{S}^{2,k}$.
  \item Consider the same set up as Item $(2)$. Then, in a neighborhood of $g_0$, the set of all $g\in\mathcal{M}^k$ which preserves double eigenvalues, i.e. $L_g$ admits an eigenvalue $\lambda(g)$ of multiplicity $2$ such that $\lambda(g)=\lambda(g_0)$, form a submanifold of meager codimension $3$ in $\mathcal{S}^{2,k}$.
\end{enumerate}
\end{theorem}

We point out that Item $(2)$ of the previous theorem generalize the Micheletti and Pistoia's result in~\cite{Micheletti}.

\vskip .3cm
\noindent\textbf{Acknowledgements:} The authors would like to express their sincere thanks to Cleon Barroso for useful comments, discussions and constant encouragement as well as to Department of Mathematics at Lehigh University, where part of this work was carried out. The first author is grateful to Huai-Dong Cao and Mary Ann for the warm hospitality and their constant encouragement.


\begin{thebibliography}{99}

\bibitem{Berger} M. Berger, Sur les premi\`eres valeurs propres des vari\'et\'es Riemanniennes. Compos. Math. 26 (1973) 129-149.

\bibitem{ahmad} A. El Soufi and S. Ilias, Domain deformations and eigenvalues of the Dirichlet Laplacian in a Riemannian manifold. Illinois J. Math. 51 (2) (2007) 645-666.

\bibitem{GMR} J.N.V. Gomes, M.A.M. Marrocos and R.R. Mesquita, Hadamard Type Variation Formulas for the Eigenvalues of the $\eta$-Laplacian and Applications. arXiv:1510.07076[math.DG].

\bibitem{henry} D.B. Henry, Perturbation of the boundary in boundary-value problems of partial differential equations. Cambridge University Press, 2005.

\bibitem{hormander} L. H\"ormander, The Analysis Partial Differential Operators III. Springer-Verlag, 1985.

\bibitem{Kato} T. Kato, Perturbation theory for linear operators. Springer, 1980.

\bibitem{lion} J.L. Lions and E. Magenes, Non-homogeneous boundary value problems and applications. Die Grundlehren der mathematischen Wissenschaften v. 181. Springer, New York 1972.

\bibitem{Marrocos} M.A.M. Marrocos and A.L. Pereira, Eigenvalues of the Neumann Laplacian in symmetric regions.  J. Math. Phys. 56 (2015) 111502.

\bibitem{Micheletti} A.M. Micheletti and A. Pistoia, Multiple eigenvalues of the Laplace-Beltrami operator and deformation of the Riemannian metric. Discrete Contin. Dyn. Syst 4 (1998) 709-720.

\bibitem{pereira2} A.L. Pereira and M.C. Pereira, An eigenvalue problem for the biharmonic operator on Z2-symmetric regions. J. London Math. Soc. 77 (2) (2008) 424-442.

\bibitem{teytel} M. Teytel, How rare are multiple eigenvalues? Comm. Pure Appl. Math. 52 (1999) 917-934.

\bibitem{Uhlenbeck} K. Uhlenbeck, Generic properties of eigenfunctions. Amer. J. Math. 98 (4) (1976) 1059-1078.

\bibitem{wall} C.T.C. Wall, Singular Points of Plane Curves. London Mathematical Society Student Texts (2004).
\end{thebibliography}
\end{document}